\documentclass[12pt]{article} 
\usepackage{amsmath, amssymb, amsthm, tikz}
\usepackage[left=2cm, right=2cm, top=2cm, bottom=2cm]{geometry}
\usepackage{standalone}
\usepackage{tikz}
\usetikzlibrary{decorations.pathreplacing}
\usepackage{graphicx}
\usepackage{caption}
\usepackage[
    colorlinks=true,
    linkcolor=black,
    citecolor=black,
    urlcolor=black
]{hyperref}
\usetikzlibrary{arrows.meta}
\usepackage{mathtools}
\DeclareMathOperator{\ssp}{ssp}

\theoremstyle{plain}
\newtheorem{theorem}{Theorem}
\newtheorem{lemma}[theorem]{Lemma}

\newtheorem{cor}[theorem]{Corollary}
\theoremstyle{definition}
\newtheorem{defn}[theorem]{Definition}
\theoremstyle{remark}

\def\aftermath{\par\vspace{-\belowdisplayskip}\vspace{-\parskip}\vspace{-\baselineskip}}

\begin{document}
\title{Separating Path Systems of Size at most $7.75n$}
\author{
Daniel W. Cranston\thanks{%
Department of Mathematics, William \& Mary, Williamsburg, VA, USA; \texttt{dcransto@gmail.com}%; 
}
\and
Jared Noyman\thanks{%
Department of Mathematics, William \& Mary, Williamsburg, VA,
USA; \texttt{jnoyman04@gmail.com}. This work was supported by the William \& Mary Charles Center Undergraduate Summer Research Grant.%;
}
\and
Gexin Yu\thanks{%
Department of Mathematics, William \& Mary, Williamsburg, VA,
USA; \texttt{gyu@wm.edu}%; 
}
}
\date{October 5, 2026}

\maketitle

\begin{abstract} \small
A family of paths in a graph $G$ strongly separates the edges of $G$ if for every ordered pair of distinct edges $(e,f)$ some path in the family contains $e$ and avoids $f$; the minimum size of such a family is denoted by $\ssp(G)$.
Bonamy, Botler, Dross, Naia, and Skokan proved,
for every $n$-vertex graph $G$, 
that $\ssp(G)\le 19n$;
Liu, Xu, and Yang recently improved this to $\ssp(G)\le 10n-o(n)$. We prove,
for every $n$-vertex graph $G$, 
that $\ssp(G)\le 7.75n$.
\end{abstract}

\section{Introduction}

Let $G$ be a graph. A family $\mathcal{P}$ of paths in $G$ \emph{strongly separates} the edges of $G$ if for any two distinct edges $e,f\in E(G)$ there is a path in $\mathcal{P}$ containing $e$ but not $f$, and a path in $\mathcal{P}$ containing $f$ but not $e$. We call such a family a \emph{strongly separating path system} of $G$, and write $\ssp(G)$ for the minimum size of one. If only one of the two directions is required, the family is \emph{weakly separating}. Separating systems were first studied by R\'enyi~\cite{renyi1961random}. In 2013, Katona  (see~\cite{letzter2024separating})
sought to determine the minimum number of paths needed to weakly separate the edges of an arbitrary $n$-vertex graph.

Falgas-Ravry, Kittipassorn, Kor\'andi, Letzter, and Narayanan~\cite{falgas2013separating} conjectured that every $n$-vertex graph admits a weakly separating path system of size $O(n)$, and verified this for several classes of graphs, including graphs of linear minimum degree and random graphs. Balogh, Csaba, Martin, and Pluh\'ar~\cite{balogh2016path} independently proposed the strong version of the conjecture and proved $\ssp(G)=O(n\log n)$. Letzter~\cite{letzter2024separating} improved this to $O(n\log^\star n)$ using sublinear expanders, and Bonamy, Botler, Dross, Naia, and Skokan~\cite{bonamy2023separating} settled the conjecture by proving $\ssp(G)\le 19n$. Their argument repeatedly deletes the set $S$ of rotation endpoints (defined in Section~\ref{Posa Section}) of a longest path and separates the edges incident with $S$ by paths routed along this longest path. Liu, Xu, and Yang~\cite{liu2026strongly} refined each component of this approach and proved that $\ssp(G)\le 10n-o(n)$.

Sharper bounds are known for special classes of graphs. For complete graphs, Fernandes, Mota, and Sanhueza-Matamala~\cite{fernandes2025separating} proved $\ssp(K_n)=(1+o(1))n$, and Kontogeorgiou and Stein~\cite{kontogeorgiou2024exact} obtained $\ssp(K_n)\le n+9$. Fernandes, Hoppen, Kontogeorgiou, Mota, and Peng~\cite{fernandes2026separating} proved $\ssp(G)\le n$ for every $2$-degenerate $n$-vertex graph, together with bounds for subcubic, planar, and complete bipartite graphs. For a lower bound, a disjoint union of copies of $K_4$ requires $5n/4$ paths~\cite{liu2026strongly}; so no bound of the form $(1+o(1))n$ holds for all graphs. The correct universal constant thus lies in $[5/4,10]$. From now on, we refer to strong separation as separation. We prove the following result.

\begin{theorem}\label{7.75n}
    Every graph on $n$ vertices admits a separating path system of size at most $7.75n$.
\end{theorem}

The proof of Theorem~\ref{7.75n} follows the framework used in~\cite{liu2026strongly}, which is similar to that of~\cite{bonamy2023separating}. We start with an $n$-vertex graph $G$, let $G_0 \coloneq G$, and iterate the following steps until no edges remain.
(We provide more details in Section~\ref{Proof of Main Result}.  In particular, see Figure~\ref{fig:separation_figure}.)

\begin{itemize}
    \item In the current graph $G_i$, fix a longest path $P_i$ and one of its endpoints $v$, and let $S_i$ be the set of endpoints of longest paths obtained from $P_i$ by P\'osa rotations fixing $v$ (see Section~\ref{Posa Section}).
    \item Partition the edges incident to $S_i$ into those lying on $P_i$ and those not, the latter forming a subgraph $H_i$.
    \item Separate the edges of $H_i$ by constructing a family of paths in $H_i\cup P_i$ (see Section~\ref{Contracting P}).
    \item Apply a path covering theorem of Fan to the components of $H_i$, obtaining a family of paths containing each edge of $H_i$. These paths separate the edges of $H_i$ from the edges of $P_i$, at a cost of at most $(3|S_i|-1)/2+|S_i|/4$ paths.
    \item Let $G_{i+1}\coloneq G_i\setminus S_i$.
\end{itemize} 
The edges of $P_i$ incident to $S_i$ are set aside at each iteration, and their union over all iterations induces a $2$-degenerate subgraph $B$ of $G$, whose edges can be separated by at most $n$ paths.

In~\cite{liu2026strongly}, the edges of $H_i$ are separated by constructing paths in $H_i\cup P_i$, using subpaths of $P_i$ to connect edges of $H_i$. We refer to this procedure as routing along $P_i$. This requires at most $5(3|S_i|-1)/2$ paths. Our improvement comes from contracting $P_i$ before routing which reduces the bound on the number of vertices used in the routing from $3|S_i|-1$ to $2|S_i|$, hence reducing the number of paths constructed to at most $5|S_i|$. These paths are routed along the contracted path $P_i^*$, and we show that each path constructed this way lifts to a path in $H_i\cup P_i$ whose edges in $H_i$ are precisely the preimages of its edges in $H_i^*$. Since the deleted vertex subsets $S_i$ are pairwise disjoint, $\sum_{i\ge0} |S_i|\leq n$. Summing over all iterations therefore gives at most $5n$ paths separating edges within each $H_i$, and at most $7n/4$ paths covering the edges of each $H_i$. Because each path used to cover the edges of $H_i$ lies entirely in $H_i$, these paths separate edges belonging to different $H_i$'s. Together with the separating path system contained in $B$, the edges of each $H_i$ are separated from the edges of $B$ in both directions. Combining these families proves Theorem~\ref{7.75n}, as the total number of paths separating the edges of $G$ is at most
\[
    5n+7n/4 +n=31n/4=7.75n.
\]

\section{Preliminaries}\label{Preliminatires}
In the proof of Theorem~\ref{7.75n} (see section~\ref{Proof of Main Result}) we use Theorems~\ref{2-degenerate}, \ref{Fan}, and \ref{Separate H} without proof. Here we provide a brief proof sketch of each. We also prove an easy corollary of Theorem~\ref{2-degenerate} giving an improvement in the coefficient of Theorem~\ref{7.75n} for $d$-degenerate graphs with $d\le 14$.

\begin{theorem}[Fernandes, Hoppen, Kontogeorgiou, Mota, and Peng~\cite{fernandes2026separating}] \label{2-degenerate}
    Every $2$-degenerate graph on $n$ vertices admits a separating path system of size at most $n$.
\end{theorem}

Note that we can choose the separating path system in Theorem~\ref{2-degenerate} to cover every edge. When the graph has at least two edges, this follows directly from strong separation; when the graph has exactly one edge, we take the path consisting of that edge. In~\cite{fernandes2026separating}, the authors first reduce Theorem~\ref{2-degenerate} to the connected case. For connected graphs on at least three vertices, their inductive construction maintains two invariants of the separating system: every edge belongs to exactly two paths, and every vertex is an endpoint of exactly two paths. They remove a vertex of minimum degree and apply induction to the resulting components. They then modify the resulting paths to incorporate the removed vertex and its incident edges while preserving separation and the two invariants. By decomposing a $d$-degenerate graph into $\lceil d/2\rceil$ 2-degenerate graphs, and applying Theorem~\ref{2-degenerate} to each, we prove the following corollary.

\begin{cor} \label{d-degenerate}
    Fix $d\ge 1$. The edges of every $d$-degenerate graph on $n$ vertices for $d\ge 1$ can be separated by at most $\lceil d/{2}\rceil n$ paths.
\end{cor}
\begin{proof}
    We claim that the edges of any $d$-degenerate graph can be partitioned into $\lceil{d}/{2}\rceil$ spanning edge-disjoint $2$-degenerate subgraphs. Let $r\coloneq\lceil{d}/{2}\rceil$. We proceed by induction on $n$. For $n=1$ take $r$ edgeless spanning subgraphs. Now suppose that $n\ge 2$ and the claim is true for all $d$-degenerate graphs on fewer than $n$ vertices. Let $G$ be a $d$-degenerate graph on $n$ vertices. Delete a vertex $v\in V(G)$ of minimum degree. Since $G$ is $d$-degenerate, $d_G(v)\le d$, and the new subgraph $G-v$ is also $d$-degenerate. By induction $G-v$ can be partitioned into $r$ spanning edge-disjoint $2$-degenerate subgraphs $H_1,\ldots,H_r$. We partition the edges incident to $v$ into sets $F_1,\ldots,F_r$, each of size at most $2$, allowing empty sets. This is possible because $\deg_G(v)\leq d\leq 2r$. For each $i\in [r]$ where $[r]\coloneq \{1,\ldots, r\}$, let $G_i$ be the graph formed from $H_i$ by adding $v$ and the edges in $F_i$. 

   We claim that each $G_i$ is $2$-degenerate. Indeed, every non-empty subgraph of $G_i$ containing $v$ has a vertex of degree at most $2$, namely $v$, while every non-empty subgraph avoiding $v$ is a subgraph of some $H_i$, and thus has a vertex of degree at most $2$. Furthermore, $E(G_1),\ldots, E(G_r)$ partition $E(G)$. This completes the induction

   Since each $G_i$ is a $2$-degenerate $n$-vertex graph, by Theorem~\ref{2-degenerate}, each $G_i$ admits a separating path system $\mathcal{P}_i$ of size at most $n$. We claim that $\mathcal P\coloneq\bigcup_{i=1}^r\mathcal P_i$ is a separating path system of $G$. Indeed, edges in the same $G_i$ are separated by $\mathcal P_i$. For all distinct $i,j\in [r]$, and edges $e\in E(G_i)$ and $f\in E(G_j)$ there exists a path of $\mathcal{P}_i$ containing $e$ and avoiding $f$ and symmetrically a path of $\mathcal{P}_j$ containing $f$ and avoiding $e$. Thus, $\mathcal P$ separates the edges of $G$, and
\[
|\mathcal P|
\leq \sum_{i=1}^r|\mathcal P_i| \leq rn =\lceil {d}/2\rceil n.
\]
\aftermath
\end{proof}

Another ingredient in our proof of Theorem~\ref{7.75n} is the following path covering theorem of Fan. A \emph{path cover} of a graph $G$ is a family of paths in $G$ whose union contains every edge of $G$. Note that every separating path system of a graph with at least two edges is also a path cover.

\begin{theorem} [Fan~\cite{fan2002subgraph}] \label{Fan}
    Every connected $n$-vertex graph admits a path cover of size at most $\lceil{n}/{2}\rceil$.
\end{theorem}

Fan's proof of Theorem~\ref{Fan} assumes the theorem is false and takes, for fixed $n$, $G$ to be a counterexample with maximum degree as large as possible. An edge-switching operation on $G$ yields a connected graph $G'$ on the same vertex set with strictly larger maximum degree. By the choice of $G$, this new graph $G'$ admits a path cover of size at most $\lceil n/2\rceil$. Fan then shows that this path cover of $G'$ can be transferred back to $G$ without increasing the number of paths, giving a contradiction. The following theorem is given as Lemma 2.4 in~\cite{liu2026strongly}.

\begin{theorem}[Liu, Xu, and Yang~\cite{liu2026strongly}]\label{Separate H}
     Fix a graph $G$.
     Let $F$ be a subgraph on $r$ 
     non-isolated vertices, with $r\ge 2$, and let $Q$ be a path such that $V(F)\subseteq V(Q)$ and $E(F)\cap E(Q)= \emptyset$. Then there exists a family of paths in $F \cup Q$ that strongly separate the edges of $F$ with size at most ${5r}/{2}$.

\end{theorem} 

 In \cite{liu2026strongly}, when $r$ is odd, the authors label the vertices of $F$ by $0,\dotsc,r-1$ according to their order along $Q$, viewing these labels as elements of $\mathbb{Z}_r$. They then partition $E(F)$ into sum classes, grouping edges according to $i+j\pmod r$, where $i$ and $j$ are the labels of their endpoints. A second partition of $E(F)$ groups the edges according to the cyclic distance $\min\{|i-j|,r-|i-j|\}$ between their endpoint labels. The edges of each sum class are covered by a single path, while the edges of each distance class are covered by at most three paths. Each path constructed this way is routed along $Q$ and contains no edges of $F$ outside its assigned class. For odd $r$, the sum and distance classes jointly identify each edge uniquely, so the resulting paths strongly separate the edges of $F$. This gives at most $r+3\big (({r-1})/{2}\big)={(5r-3)}/{2}$ paths. The authors treat even $r$ separately, using a similar strategy and obtaining at most ${5r}/2$ paths. 

\section{P{\'{o}}sa Rotation Extension} \label{Posa Section}

Let $G$ be a graph and let $P=u,\dotsc,v$ be a longest path of $G$. For a vertex set $S\subseteq V(G)$, we define its
\textit{neighborhood} in $G$ by
\[
N_G(S)\coloneq\{y\in V(G)\setminus S:
xy\in E(G)\text{ for some }x\in S\}.
\] Let $x\in N_G(u)$ be a neighbor of $u$, and note that $x$ lies on $P$ by the maximality of $P$. Let $w$ be the predecessor of $x$ in the order of $P$ from $u$ to $v$. Let $P'\coloneq P\cup{\{ux\}}\setminus{\{wx\}}$. We say that $P'=w,\dotsc,u,x,\dotsc,v$ is \emph{derived} from $P$ by an \emph{elementary edge exchange fixing $v$} and we call $w$ the \emph{free endpoint} of $P'$.  
Also note that $V(P')=V(P)$. We give the following standard lemma and its proof to illustrate techniques we use later in Sections~\ref{Contracting P} and~\ref{Proof of Main Result}. Readers familiar with P\'osa
rotations and Lemma~\ref{Posa} may skip to Section~\ref{Contracting P}.

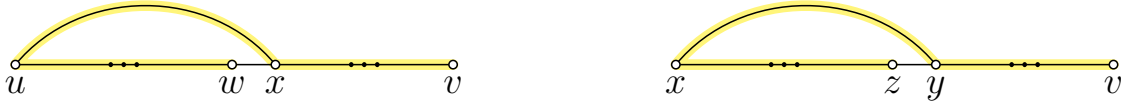
\begin{figure}[htbp]
    \centering
    \resizebox{0.85\linewidth}{!}{\begin{tikzpicture}[yscale=1.1,
  line width=1pt,
  line cap=round,
  line join=round,
  vertex/.style={circle,draw=black,fill=white,inner sep=0pt,minimum size=5.5pt},
  every label/.style={font=\large}
]
  \coordinate (u) at (0,0);
  \coordinate (w) at (5,0);
  \coordinate (x) at (6.0,0);
  \coordinate (v) at (10.1,0);

  \draw[yellow!65,line width=7pt]
    (w)--(u) .. controls (1.5,1.65) and (4.5,1.65) .. (x)--(v);

  \draw[black] (0,0)--(10.1,0);
  \draw[black] (u) .. controls (1.5,1.65) and (4.5,1.65) .. (x);

  \foreach \xx in {2.20,2.50,2.80,7.75,8.05,8.35}
    \fill (\xx,0) circle (1.5pt);

  \node[vertex,label=below:{\huge{$u$}}] at (u) {};
  \node[vertex,label=below:{\huge{$w$}}] at (w) {};
  \node[vertex,label=below:{\huge{$x$}}] at (x) {};
  \node[vertex,label=below:{\huge{$v$}}] at (v) {};

  \begin{scope}[xshift = 6in]
  \coordinate (x) at (0,0);
  \coordinate (z) at (5,0);
  \coordinate (y) at (6.0,0);
  \coordinate (v) at (10.1,0);

  \draw[yellow!65,line width=7pt]
    (z)--(x) .. controls (1.5,1.65) and (4.5,1.65) .. (y)--(v);

  \draw[black] (0,0)--(10.1,0);
  \draw[black] (x) .. controls (1.5,1.65) and (4.5,1.65) .. (y);

  \foreach \xx in {2.20,2.50,2.80,7.75,8.05,8.35}
    \fill (\xx,0) circle (1.5pt);

  \node[vertex,label=below:{\huge{$x$}}] at (x) {};
  \node[vertex,label=below:{\huge{$z$}}] at (z) {};
  \node[vertex,label=below:{\huge{$y$}}] at (y) {};
  \node[vertex,label=below:{\huge{$v$}}] at (v) {};
  \end{scope}
\end{tikzpicture}}
    \captionsetup{width=.8\linewidth}
    
    \caption{Left: A longest path derived from another longest path $P$ via an elementary exchange fixing $v$. Right: The path $Q'$ derived from $Q$ in the proof of Lemma~\ref{Posa}.}
    \label{Rotation_figure}
\end{figure}

\begin{lemma} [\cite{posa1976hamiltonian}]\label{Posa}
   Let $P=u,\dotsc, v$ be a longest path of a graph $G$ and let $S_v(P)$ denote the set of endpoints of all possible longest paths derived from $P$ by elementary exchanges fixing $v$. For short, we let $S\coloneq S_v(P)$ and let $s\coloneq|S|$. Then $N_P(S)=N_G(S)$ and $|N_G(S)| \le 2s-1$.
\end{lemma}

\begin{proof} Because $P$ is a path, its maximum degree is $2$. By definition, each $y \in N_P(S)$ is adjacent to some $x \in S$ via an edge of $P$. Since each $x \in S$ contributes at most $2$ neighbors on $P$ and $u$ contributes at most $1$, we have $|N_P(S)|\le 2s-1$. Thus, it suffices to prove that $N_G(S) \subseteq N_P(S)$.

Let $y$ be an arbitrary vertex in $N_G(S)$. By definition, there exists a vertex $x \in S$ such that $xy \in E(G)$. If $xy\in E(P)$, then there is nothing to prove. Otherwise, since $x \in S$, there exists a longest path $Q$ ending at $x$ that was derived from $P$ via a finite sequence of elementary exchanges, fixing $v$ and $Q$ is ordered from $x$ to $v$. By the maximality of $Q$, we have $y \in V(Q)$; otherwise, the path $Q \cup \{xy\}$ would be longer than $Q$, a contradiction. 

If $xy\notin E(Q)$, let $z$ be the predecessor of $y$ in the order of $Q$ and let $Q' \coloneq Q \cup \{xy\} \setminus \{yz\}$ be a longest path ending at $z$ derived from $Q$ and fixing $v$ (see Figure~\ref{Rotation_figure}). By definition, $z \in S$. Consider the sequence of derivations from $P$ to $Q'$:

$$P = P_0 \xrightarrow{\text{exchange 1}} P_1 \xrightarrow{\text{exchange 2}} \cdots \xrightarrow{\text{exchange } k} P_k = Q \xrightarrow{\text{exchange } k+1} P_{k+1} = Q'.$$

If $xy \in E(Q)$, consider instead the sequence
$P=P_0,\dots,P_k=Q$. Since $xy \notin E(P)$, some exchange in this sequence added the edge $xy$. Moreover, since $y \notin S$, the path immediately before this exchange in the sequence had $x$ as its free endpoint. Thus, this exchange removed the edge joining $y$ to its predecessor on that path.

Notice the following:

\begin{itemize}
    \item [(1)] Let $e$ be the first edge incident to $y$ that is removed during the sequence of derivations. This edge must exist by the above argument if $xy\in E(Q)$, or the final step ($Q \to Q'$) explicitly removes $yz$ if $xy\notin E(Q)$. Because $e$ is the first edge incident to $y$ to be deleted, it must be contained in the initial path.
    \item [(2)] If an edge is removed at any step of this sequence, then one of its 
    endpoints must lie in $S$. \end{itemize}

Therefore, by (1), $e \in E(P)$ and by (2), one endpoint of $e$ must lie in $S$. By construction, $y\notin S$, so its neighbor via $e$ must be in $S$. It follows that $y\in N_P(S)$ and since $N_P(S) \subseteq N_G(S)$ we have $N_P(S) = N_G(S)$.
\end{proof}

\section{Contracting \texorpdfstring{$P$}{P}} \label{Contracting P}

In this section, we identify the subpaths of $P$ that we will contract, and we establish the contraction and lifting lemmas that we will use in Section~\ref{Proof of Main Result}.

\begin{lemma} \label{Segment}
Let $P=u,\dots,v$ be a longest path and define $S$ as in Lemma~\ref{Posa}. If $W$ is a subpath of $P$ with $V(W)\cap S=\emptyset$, then for any longest path $Q$ derived from $P$, by a sequence of elementary exchanges fixing $v$, we have $E(W)\subseteq E(Q)$. Furthermore, $W$ occurs as a subpath of $Q$.
\end{lemma}

Intuitively, Lemma~\ref{Segment} holds because removing any edge $e$ of $W$ in an elementary exchange fixing $v$ puts at least one endpoint of $e$ into $S$, contradicting our hypothesis that $V(W)\cap S = \emptyset$. Below we formalize this intuition.

\begin{proof}
Since $Q$ is derived from $P$ by elementary exchanges fixing $v$, there is a sequence $Q_0,\dots,Q_r$, with $Q_0=P$ and $Q_r=Q$, in which each path is derived from the previous one by a single elementary exchange fixing $v$. Let $Q_t=q_0,\dots,q_k$, with $q_k=v$, be a path realized at some point in this sequence. An exchange applied to $Q_t$ using a chord $q_0q_i$ with $i\ge2$ deletes exactly the edge $q_{i-1}q_i$ and makes $q_{i-1}$ the new free endpoint; hence $q_{i-1}\in S$. Therefore, every edge deleted in an exchange has an endpoint in $S$.

Since no edge of $W$ has an endpoint in $S$, no edge of $W$ is ever deleted, and induction along the sequence gives $E(W)\subseteq E(Q_t)$ for every $t\in \{0,\ldots, r\}$. In particular $E(W)\subseteq E(Q)$. Finally, $W$ is connected, so $E(W)$ spans a connected subgraph of the path $Q$, which is therefore a subpath of $Q$ with the same ends as $W$. Thus, $W$ occurs in $Q$ as a subpath.
\end{proof}

\begin{defn}\label{segment}
Let $P=u,\dots,v$ be a longest path and define $S$ as in Lemma~\ref{Posa}. A \emph{segment} of $P$ is a component of $P\setminus S$, that is, a maximal subpath of $P$ containing no vertex of $S$. We write $P[x,y]$ for a segment with ends $x$ and $y$, labeled so that $x$ precedes $y$ in the order of $P$ and often let $W\coloneq P[x,y]$. Note that a segment consisting of a single vertex has $x=y$. We list the segments as $W_1,\dots,W_m$ in their order along $P$. 
\end{defn}

\begin{lemma} \label{disjoint N}
    Fix a segment $P[x,y]$ of $P$. If we let $W\coloneq P[x,y]$, then
        \begin{itemize}
            \item[(i)] $\big(V(W)\setminus\{x,y\}\big)\cap N_G(S)=\emptyset$; and
            \item[(ii)] if $x\neq y$, then $N_G(x)\cap N_G(y)\cap S=\emptyset$.
        \end{itemize}
\end{lemma}

\begin{proof}
For (i) let $w$ be an interior vertex of $W$. Both neighbors of $w$ on $P$ lie in $V(W)$, hence outside $S$, so $w\notin N_P(S)$ and by Lemma~\ref{Posa}, $N_G(S)=N_P(S)$, so $w\notin N_G(S)$. 

As for (ii), suppose for a contradiction that there exists $z\in S$ such that $xz,yz\in E(G)$. Since $z\in S$, there exists some path $R=z,\dotsc,v$ that is derived from $P$ by elementary exchanges fixing $v$. By Lemma~\ref{Segment}, $W$ is a subpath of $R$; by symmetry we assume that $x$ preceeds $y$ on $R$ (otherwise, we swap the names $x$ and $y$). We can write $R=z,\dotsc,x,\dotsc,y, \dotsc,v$. But since $yz$ is an edge, $y^-$ the predecessor of $y$ in the order of $R$ must be in $S$ since a longest path $y^-, \dotsc,x,\dotsc, z, y, \dotsc, v$ can be derived from $R$ by an elementary exchange fixing $v$ (see Figure~\ref{fig:Segment_figure}). But then $y^-$ is a vertex in $S$ that lies in $W$, contradicting the hypothesis that $W$ is a segment. 
\end{proof}

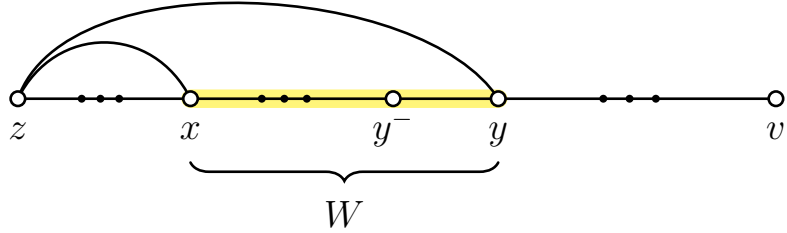
\begin{figure}[htbp]
    \centering
    \resizebox{0.6\linewidth}{!}{\begin{tikzpicture}[
  line width=1.0pt,
  vertex/.style={circle,draw,fill=white,inner sep=0pt,minimum size=5.5pt},
  every label/.style={font=\large,text height=1.5ex,text depth=0.25ex},
  line cap=round
]
  \coordinate (z) at (0,0);
  \coordinate (a) at (2.3,0);
  \coordinate (b) at (5.0,0);
  \coordinate (c) at (6.4,0);
  \coordinate (v) at (10.1,0);

  \draw[yellow!65,line width=7pt] (a)--(c);

  \draw[black,solid,line width=1pt] (0,0)--(10.1,0);

  \draw (z) .. controls (0.50,1.0) and (1.80,1.0) .. (a);
  \draw (z) .. controls (0.65,1.85) and (5.40,1.55) .. (c);

  \foreach \xx in {0.85,1.10,1.35,3.25,3.55,3.85,7.80,8.15,8.50}
    \fill (\xx,0) circle (1.5pt);

  \node[vertex,label=below:{$z$}] at (z) {};
  \node[vertex,label=below:{$x$}] at (a) {};
  \node[vertex,label=below:{$y^{-}$}] at (b) {};
  \node[vertex,label=below:{$y$}] at (c) {};
  \node[vertex,label=below:{$v$}] at (v) {};

  \draw[decorate,decoration={brace,mirror,amplitude=7pt}]
    (2.3,-0.85)--(6.4,-0.85)
    node[midway,below=10pt,font=\large] {$W$};
\end{tikzpicture}}
    \captionsetup{width=.6\linewidth}
    \caption{An illustration of Lemma 9 with a longest path $R$ ending at $z$ with the edges $zx$ and $zy$ in $E(G)$.}
    \label{fig:Segment_figure}
\end{figure}

Informally, we aim to contract each segment of $P$ to a single vertex, construct a family of paths in the graph formed by contracting $H\cup P$, and then lift these paths to paths in $H\cup P$ that separate the edges of $H$. We now formalize this approach.

Let $P=u,\dots,v$ and define $S$ and $s$ as in Lemma~\ref{Posa}. Let $E_S$ be the edges of $G$ with at least one endpoint in $S$ and let $B\coloneq E_S\cap E(P)$. Let $H$ be the subgraph of $G$ that is edge-induced by $E_S\setminus B$. Note that $B$ and $E(H)$ partition $E_S$. Recall by Lemma~\ref{Posa} that $V(H)\subseteq S\cup N_G(S) \subseteq V(P)$ (see the top of Figure~\ref{fig:contraction}). Let $W_1, \dotsc, W_m$ be the segments of $P$ as in Definition~\ref{segment}. Since $W_1, \dotsc, W_m$ are the components of $P\setminus S$, each vertex in $N_G(S)$ is an end of some segment by Lemma~\ref{disjoint N} (i).

\begin{defn} \label{Definition of phi}
     Let $\varphi$ be the map on $V(P)$ that fixes every vertex of $S$ and sends every vertex of the segment $W_j$ to a single new vertex $w_j$, for each $j\in [m]$; see Figure~\ref{fig:contraction}. Since $V(H) \subseteq V(P)$, it follows that $\varphi$ is defined on all of $V(H)$. We write $P^*$ for the image of $P$ under $\varphi$, that is, 
\[
V(P^*)=\varphi(V(P)) \text{ and }E(P^*)= \{\varphi(a)\varphi(b): ab\in E(P),\text{ and } \varphi (a)\not= \varphi (b)\}.
\] 
We write $H^*$ for the image of $H$ under $\varphi$, with 
\[
V(H^*)=\varphi(V(H)) \text{ and } E(H^*)=\{\varphi(a)\varphi(b):ab\in E(H)\}.
\]
\end{defn}

\begin{lemma}[Contraction] \label{Contract}
The following 3 statements hold.
\begin{itemize}

\item[(i)] $P^*$ is a path on at most $2s$ vertices, i.e. $|V(P^*)| \le 2s$. 
\item [(ii)]$\varphi$ induces a bijection from $E_S$ to $E(H^*)\cup E(P^*)$ mapping $E(H)$ onto $E(H^*)$ and $B$ onto $E(P^*)$. 

\item[(iii)] $H^*$ is a simple graph with $V(H^*) \subseteq V(P^*)$ and $E(H^*)\cap E(P^*)=\emptyset$.

\end{itemize}

\end{lemma}

\begin{proof}

For (i), contracting vertex-disjoint subpaths of a path yields a path, so $P^*$ is a path. Furthermore, each segment has a unique vertex of $S$ immediately preceding it on $P$. Hence $m\le s$ and $|V(P^*)|=s+m\le 2s$.

For (ii), we first prove surjectivity. By Definition~\ref{Definition of phi}, we have $E(H^*)=\{\varphi(a)\varphi(b):ab\in E(H)\}$, so every edge of $H^*$ is the image of an edge of $E(H)$. 
Similarly, we have $E(P^*)=\{\varphi(a)\varphi(b): ab\in E(P) \text{ and } \varphi(a)\neq \varphi(b)\}$, which implies that every edge of $P^*$ is the image of an edge $ab\in E(P)$ with $\varphi(a)\neq\varphi(b)$ (see Figure~\ref{fig:contraction}). 
Thus, $\varphi$ is surjective.
As the segments are maximal subpaths of $P$ avoiding $S$, two vertices that are adjacent on $P$ lying outside $S$ belong to a common segment and hence have the same image under $\varphi$. Therefore, if $\varphi(a)\varphi(b)\in E(H)$, then at least one of $a$ and $b$ lies in $S$, so $ab\in B$. Thus, $\varphi$ maps $E(H)$ onto $E(H^*)$ and $B$ onto $E(P^*)$. 

Now we show that $\varphi$ is injective on $E_S$.
Recall that each edge of $E_S$ has either both endpoints in $S$, or one endpoint in $S$ and one endpoint that is the end of some segment of $P$. 
Fix distinct $e,e'\in E_S$ and suppose that $\varphi(e)=\varphi(e')$. If $e$ and $e'$ each have both endpoints in $S$, then $\varphi(e)=e$ and $\varphi(e')=e'$ since $\varphi$ fixes $S$ pointwise. But then $e=e'$, which contradicts our assumption. 
Suppose instead (by symmetry between $e$ and $e'$) that $e$ has both endpoints in $S$ and $e'$ has one endpoint in $S$ and one endpoint in some segment $W_j$ for some $j\in[m]$. Then $\varphi(e)\neq\varphi(e')$ since $\varphi(e')$ has an endpoint $w_j$ which is not in $S$. Lastly, suppose $e=zy$ and $e'=z'y'$ with $z,z'\in S$ and $y,y'$ ends of segments $W_j$ and $W_k$. Then $\varphi(e)=zw_j$ and $\varphi(e')=z'w_k$, where $z,z'\in S$ and $w_j,w_k\notin S$. If $\varphi(e)=\varphi(e')$, then $z=z'$ and $w_j=w_k$, so $j=k$ and $y,y'$ are both ends of $W_j$. If $y\neq y'$ they $y,y'$ are the two distinct ends of $W_j$, and $z\in N_G(y)\cap N_G(y')\cap S$, contradicting Lemma~\ref{disjoint N} (ii). Therefore, $y=y'$ and $e=e'$, contrary to assumption. Hence $\varphi$ is injective on $E_S$. Since $B$ and $E(H)$ partition $E_S$, their images $E(P^*)$ and $E(H^*)$ are disjoint, and $\varphi$ induces a bijection from $E_S$ to $E(H^*)\cup E(P^*)$ mapping $E(H)$ onto $E(H^*)$ and $B$ onto $E(P^*)$.

For (iii), since every edge of $E(H)$ has an endpoint in $S$ and $\varphi$ fixes $S$ pointwise, no edge of $E(H)$ has both endpoints mapped to the same vertex under $\varphi$. Therefore, $H^*$ has no loops. By (ii), $\varphi$ induces a bijection from $E(H)$ to $E(H^*)$, so $H^*$ has no multi edges and is therefore a simple graph. Moreover, $V(H^*)=\varphi(V(H))\subseteq\varphi(V(P))=V(P^*)$ and $E(H^*)\cap E(P^*)=\emptyset$ by (ii).
\end{proof}

\begin{figure}[htbp]
    \centering
    \resizebox{0.7\linewidth}{!}{\begin{tikzpicture}[
  line width=1.0pt, line cap=round, line join=round,
  vertex/.style={circle,draw,fill=white,inner sep=0pt,minimum size=6pt},
  Bedge/.style={draw=red,line width=1.8pt},
  Hedge/.style={ draw=blue!75!black, line width=1.0pt, dash pattern=on 4pt off 2pt},
  marked/.style={vertex,fill=black},
  group/.style={decorate,decoration={brace,mirror,amplitude=7pt}},
  every node/.style={font=\large}
]

  \foreach \i in {0,...,11}
    \coordinate (t\i) at ({1.2*\i},0);

  \draw (t1)--(t2) (t4)--(t6) (t9)--(t11);
  \draw[Bedge] (t0)--(t1) (t2)--(t4) (t6)--(t9);
  \draw[Hedge] (t0) .. controls (0.9,1.55) and (3.9,1.55) .. (t4);
  \draw[Hedge] (t3) .. controls (5.0,1.75) and (9.4,1.75) .. (t9);
  \draw[Hedge] (t6) .. controls (7.6,0.95) and (9.2,0.95) .. (t8);
  \foreach \i in {1,2,4,5,6,9,10,11}
    \node[vertex] at (t\i) {};
  \foreach \i in {0,3,7,8}
    \node[marked] at (t\i) {};
  \foreach \i/\lab in {0/u,3/z_1,7/z_2,8/z_3,11/v}
    \node[anchor=base] at ({1.2*\i},-0.48) {$\lab$};
  \draw[group] (1.0,-0.27)--(2.6,-0.27)
    node[midway,below=11pt] {$W_1$};
  \draw[group] (4.6,-0.27)--(7.4,-0.27)
    node[midway,below=11pt] {$W_2$};
  \draw[group] (10.6,-0.60)--(13.4,-0.60)
    node[midway,below=11pt] {$W_3$};

  \begin{scope}[yshift=-3.7cm]
    \coordinate (b0) at (0,0);     % u
    \coordinate (b1) at (1.8,0);   % center of W_1
    \coordinate (b2) at (3.6,0);   % z_1
    \coordinate (b3) at (6.0,0);   % center of W_2
    \coordinate (b4) at (8.4,0);   % z_2
    \coordinate (b5) at (9.6,0);   % z_3
    \coordinate (b6) at (12.0,0);  % center of W_3

    \draw[Bedge] (b0)--(b6);

    \draw[Hedge] (b0)
      .. controls (1.0,1.55) and (5.0,1.55) .. (b3);
    \draw[Hedge] (b2)
      .. controls (5.0,1.75) and (10.6,1.75) .. (b6);
    \draw[Hedge] (b3)
      .. controls (6.6,0.95) and (9.0,0.95) .. (b5);

    \foreach \i in {1,3,6}
      \node[vertex] at (b\i) {};
    \foreach \i in {0,2,4,5}
      \node[marked] at (b\i) {};

    \foreach \i/\lab in
      {0/u,1/w_1,2/z_1,3/w_2,4/z_2,5/z_3,6/w_3}
      \node[anchor=base,yshift=-0.48cm]
        at (b\i) {$\lab$};
  \end{scope}
\end{tikzpicture}}
    \caption{An example of $H\cup P$ (top) and its image $H^*\cup P^*$ under $\varphi$ (bottom). The solid black vertices belong to $S$. The dashed blue edges represent the edges of $H$ and $H^*$, respectively, while the thick red edges represent the edges of $B$ and $P^*$, respectively.}
    \label{fig:contraction}
\end{figure}
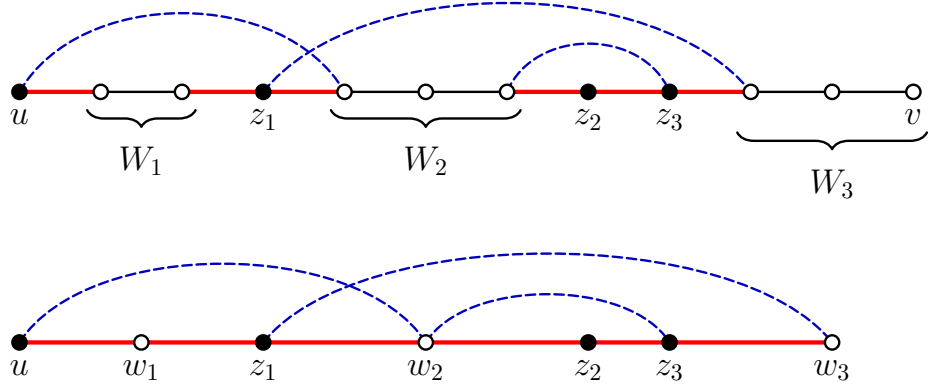

\begin{lemma}[Lifting]\label{Lifting} 
Every path $R^*$ in $H^*\cup P^*$ lifts to a path $R$ in $H\cup P$ with
  \[
  E(R)\cap E(H)=\varphi^{-1}\big(E(R^*)\cap E(H^*)\big)
  \qquad\text{and}\qquad E(R)\setminus E(H)\subseteq E(P).
  \]
\end{lemma}

\begin{proof}
Let $R^*$ be a path in $H^*\cup P^*$. By Lemma~\ref{Contract} (ii), every $f\in E(R^*)$ has a unique preimage $e\in E_S$, such that $e\in E(H)$ if $f\in E(H^*)$ and $e\in B$ if $f\in E(P^*)$. Suppose $f=zw_j$ for some $z\in S$ and $j\in[m]$. Since $\varphi$ fixes $z$, we have $e=zy$ for some $y\in V(W_j)\cap N_G(S)$, so by Lemma~\ref{disjoint N} (i), $y$ is an end of $W_j$. Thus, every edge of $R^*$ incident with some $w_j$ for $j \in [m]$ has a unique preimage in $E_S$ incident to an end of $W_j$.

We lift $R^*$ to $R$ as follows:

\begin{itemize}
  \item if $w_j$ is an end of $R^*$, replace it by the end of $W_j$ at which the preimage in $E_S$ of its unique incident edge of $R^*$ attaches;
  \item if $w_j$ is internal to $R^*$ and the preimages in $E_S$ of its two incident edges of $R^*$ attach at the same end $x$ of $W_j$, replace $w_j$ by $x$;
  \item if $w_j$ is internal to $R^*$ and the preimages in $E_S$ of its two incident edges of $R^*$ attach at opposite ends $x$ and $y$ of $W_j$, replace $w_j$ by $W_j$ traversed from $x$ to $y$.
\end{itemize}
 Every vertex of $R^*$ in $S$ is left unchanged. We claim that $R$ is a path in $H \cup P$. Indeed, consecutive vertices in $R$ are adjacent in $H\cup P$ since inserted segments are subpaths of $P$ and at every other position, the inserted vertices are the endpoints of the preimage of the corresponding edge in $R^*$, which by Lemma~\ref{Contract} (ii) lies in $E_S \subseteq E(H) \cup E(P)$. Furthermore, the vertices of $S$ in $R$ are exactly the vertices of $S$ in $R^*$ and $R^*$ is a path. The inserted segments are pairwise disjoint from each other and disjoint from $S$. Since each  $w_j$ appears at most once in $R^*$, and its replacement segment in $R$ contains no repeated vertices, no vertex of $R$ is repeated. Thus, $R$ is a path. 

 Finally, every edge of $R$ is either an edge of some $W_j$, hence an edge of $P$, or is the unique preimage of an edge of $R^*$, which by Lemma~\ref{Contract} (ii) lies in $E(H)$ if that edge lies in $E(H^*)$ or in $B$ otherwise. As $E(H)$ and $E(P)$ are disjoint, it follows that $E(R)\cap E(H)=\varphi^{-1}\big(E(R^*)\cap E(H^*)\big)$ and $E(R)\setminus E(H)\subseteq E(P)$.
\end{proof}

\section{Proof of Theorem~\ref{7.75n}} \label{Proof of Main Result}

\begin{proof} [Proof of Theorem~\ref{7.75n}]

 Let $G$ be a graph on $n$ vertices, let $G_0 \coloneq G$, and let $n_0 \coloneq n$. Starting at $i=0$, we iterate the following steps until no edges remain. Begin with a graph $G_i$ on $n_i$ vertices and let $P_i=u,\dotsc,v$ be a longest path of $G_i$. Let $S_i$ be the set of P\'osa rotation endpoints fixing $v$ as defined in Lemma~\ref{Posa}. Let $E_{S_i}$ be the set of edges of $G_i$ with at least one endpoint in $S_i$ and let $B_i \coloneq E_{S_i}\cap E(P_i)$. Let $H_i$ be the subgraph of $G_i$ edge-induced by $E_{S_i}\setminus B_i$; hence, $V(H_i) \subseteq S_i \cup N_{G_i}(S_i)$. Note that $E(H_i)$ and $B_i$ partition $E_{S_i}$. Let $G_{i+1}\coloneq G_i\setminus S_i$. 
 
 Let $s_i \coloneq |S_i|$ and note that the elements of $\{S_i\}_{i\ge 0}$ are pairwise disjoint; hence $\sum_{i\ge 0}s_i \leq n$. Let $h_i \coloneq |V(H_i)|$ and by Lemma~\ref{Posa}, we have $h_i\leq s_i+|N_{G_i}(S_i)|\leq 3s_i-1$. Let $B$ be the graph edge-induced by $\bigcup_{i\ge 0} B_i$ and note that $\{E(H_i)\}_{i\geq 0}$ and $E(B)$ partition the edges of $G$. Let $H \coloneq \bigcup_{i\ge 0} H_i$. We refer to the $i$-th iteration as the steps outlined above carried out on $G_i$. Having defined the iterations, it suffices to construct $\mathcal{B}$, $\{\mathcal{C}_i\}_{i\ge 0}$, and $\{\mathcal{P}_i\}_{i\ge 0}$ as in (i), (ii), and (iii) below.

 \begin{enumerate}
\item [(i)] A family $\mathcal{B}$ of paths contained in $B$ that separate the edges of $B$ with $|\mathcal{B}|\le n$; and
\item [(ii)] for each $i$, a family of paths $\mathcal{C}_i$ contained in $H_i$ that cover $E(H_i)$ with $|\{\mathcal{C}_i\}|\le 7n/4$; and
\item [(iii)] for each $i$, a family of paths $\mathcal{P}_i$ that separate the edges of $H_i$
with $|\{\mathcal{P}_i\}|\le 5n$.
\end{enumerate}
Now we are done since
$$|\mathcal{B}| + \sum_i|\mathcal{C}_i|+\sum_i|\mathcal{P}_i| \le n+7n/4+5n = 7.75n.$$ 

Let $\mathcal{C}\coloneq \bigcup_{i\ge 0} \mathcal{C}_i$ and $\mathcal{P}\coloneq \bigcup_{i\ge 0} \mathcal{P}_i$. We claim that $\mathcal{B} \cup \mathcal{C} \cup \mathcal{P}$ is a separating system of $E(G)$. Indeed, let $e,f$ be distinct edges of $G$ and recall that $E(B)$ and $\{E(H_i)\}_{i\ge 0}$ partition $E(G)$. If $e,f\in E(H_i)$ for some $i$, then $e$ and $f$ are separated in both directions by $\mathcal{P}_i$. If $e,f\in E(B)$, then $e$ and $f$ are separated in both directions by $\mathcal{B}$. Otherwise $e$ and $f$ lie in distinct parts, and a path containing $e$ and avoiding $f$ suffices, with the reverse following by symmetry. If $e\in E(H_i)$, then some path of $\mathcal{C}_i$ contains $e$ and is contained entirely in $H_i$, so it avoids $f\notin E(H_i)$. If $e\in E(B)$, then some path of $\mathcal{B}$ contains $e$ and is contained in $B$, so it avoids $f\notin E(B)$.

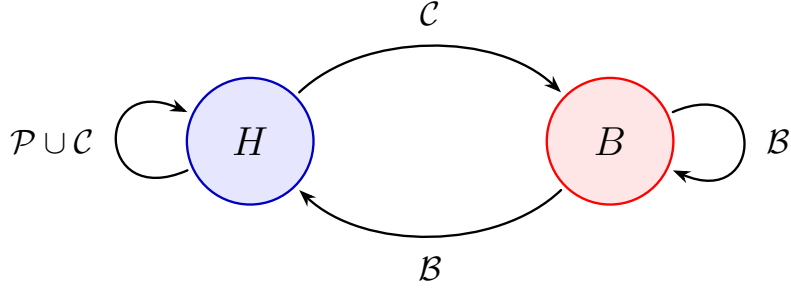
\begin{figure}[htbp]
    \centering
    \resizebox{0.6\linewidth}{!}{\begin{tikzpicture}[
  line width=1pt, line cap=round, line join=round,
  part/.style={circle,draw=black,fill=gray!15,minimum size=1.9cm,
               inner sep=0pt,font=\Large},
  arrow/.style={-{Stealth[length=2.6mm,width=1.9mm]},shorten >=2pt,shorten <=2pt},
  every node/.style={font=\large}
]
 \node[part,fill=blue!10,draw=blue!75!black]
    (H) at (0,0) {$H$};

\node[part,fill=red!10,draw=red]
    (B) at (5.4,0) {$B$};

  \draw[arrow] (H.45)
    .. controls (1.7,1.65) and (3.7,1.65) .. (B.135)
    node[midway,above=5pt] {$\mathcal{C}$};

  \draw[arrow] (B.225)
    .. controls (3.7,-1.65) and (1.7,-1.65) .. (H.315)
    node[midway,below=5pt] {$\mathcal{B}$};

  \draw[arrow] (H.205)
    .. controls (-2.35,-1.05) and (-2.35,1.05) .. (H.155)
    node[midway,left=6pt] {$\mathcal{P}\cup\mathcal{C}$};
  \draw[arrow] (B.25)
    .. controls (7.75,1.05) and (7.75,-1.05) .. (B.335)
    node[midway,right=6pt] {$\mathcal{B}$};
\end{tikzpicture}}
    \captionsetup{width=.6\linewidth}
    \caption{An illustration of the subgraphs $H$ and $B$ with the path families that separate their edges where the arrow $X \xrightarrow[]{\mathcal{F}}Y$ implies that $\mathcal{F}$ separates the edges of $X$ from the edges of $Y$.} 
    \label{fig:separation_figure}
\end{figure}

\medskip
\noindent\underline{Separating the Edges of $B$.} We claim that $B$ is $2$-degenerate. Intuitively, if we order the vertices of $B$ by the iteration in which they are deleted, every edge of $B$ joining a vertex of $S_i$ to a vertex later in order lies on $P_i$. Since $P_i$ is a path, each vertex of $S_i$ has at most $2$ later neighbors in this order. We formalize this intuition using the following argument given in \cite{liu2026strongly}. 

Order the vertices of $B$ so that the vertices of $S_i$ come before those of $S_j$ for $i<j$ and the vertices of each $S_i$ are ordered arbitrarily. Append, in arbitrary order, the vertices contained in no $S_i$, that is, the vertices remaining after the last iteration and no edges remain. We show that each vertex has at most two neighbors later in this order.

Suppose $B$ has some edge $ab$ such that $a\in S_k$ precedes $b$ in the order of $V(B)$. Since $ab\in E(B)$, we have $ab\in B_\ell$ for some iteration $\ell$. If $\ell<k$, then at least one endpoint of $ab$ must be in $S_\ell$ by definition. That endpoint cannot be $a$ since $S_k$ and $S_\ell$ are disjoint, so $b\in S_\ell$. However, $b\in S_\ell$ contradicts that $a$ comes before $b$ in order since by assumption $\ell<k$. If $\ell>k$, then $a\notin V(G_\ell)$ since $a$ was deleted at iteration $k$, hence $ab\notin E_{S_\ell}$ which contradicts that $ab\in B_\ell=E_{S_\ell}\cap E(P_\ell)$. Therefore, $\ell=k$. Thus, every edge joining $a$ to a later vertex in the order belongs to $B_k\subseteq E(P_k)$. Since $P_k$ is a path, its maximum degree is $2$ which implies that $a$ is incident to at most two edges of $B_k$. Thus, $a$ has at most two neighbors that are later in the order of $V(B)$. 

Now, let $b$ be one of the vertices appended to the order of $V(B)$, i.e. $b$ is not contained in any $S_i$. Suppose $ab \in E(B)$, such that $ab \in B_\ell$. Since $B_\ell=E_{S_\ell}\cap E(P_\ell)$ at least one endpoint of $ab$ lies in $S_\ell$. Since $b$ does not lie in any $S_i$, it follows that $b \notin S_\ell$, so $a\in S_\ell$. Hence $a$ is deleted at iteration $\ell$ and therefore precedes $b$ in the order of $V(B)$. Thus, $b$ has no later neighbor in the order of $V(B)$ and $B$ is $2$-degenerate.

Since $|V(B)| \leq n$, Theorem~\ref{2-degenerate} gives a separating path system $\mathcal{B}$ of $B$, of size at most $n$, which we choose also to cover the edges of $B$.

\medskip

\noindent\underline{Covering the Edges of $H$.}

 Fix an iteration $i$. If $H_i$ is empty, let $\mathcal{C}_i\coloneq\emptyset$. Otherwise, $H_i$ has at least two 
 non-isolated vertices. Let $C_1,\dotsc, C_c$ be the components of $H_i$. Since $H_i$ is edge-induced, no component consists of an isolated vertex. Let $r_k\coloneq|V(C_k)|$ for each $k\in [c]$ and note that $\sum_{k=1}^{c}r_k= h_i$. By Theorem~\ref{Fan}, there exists a path cover of size at most $\lceil {r_k}/{2}\rceil$ of $C_k$. Let $u_k \coloneq |V(C_k)\cap S_i|$ and note that $\sum_{k=1}^{c}u_k=|S_i\cap V(H_i)|\leq s_i$. If some $u_k=1$, then denote $V(C_k)\cap S_i$ by $ \{x\}$; now every edge of $C_k$ is incident to $x$ since $C_k$ is a component of $H_i$. Therefore, $C_k$ is a copy of $K_{1,{r_k-1}}$ and can be covered by at most $\lceil{(r_k-1)}/{2}\rceil\leq {r_k}/{2}$ paths (by pairing up the leaves). Otherwise, $u_k\geq 2$, and each such component can be covered by at most $\lceil{r_k}/{2}\rceil\leq {r_k+1}/{2}$ paths. The number of components with $u_k\geq 2$ is at most ${s_i}/{2}$. Let $\mathcal{X}$ denote the family of components of $H_i$ for which $u_k=1$ and let $\mathcal{Y}$ denote the family for which $u_k \geq 2$. Summing over all components gives a family of paths $\mathcal{C}_i$ covering $E(H_i)$ of size at most 
\[
 \sum_{C_k\in\mathcal{X}} \bigg\lceil\frac{r_k-1}{2}\bigg\rceil +\sum_{C_k\in\mathcal{Y}} \bigg\lceil\frac{r_k}{2}\bigg\rceil \leq \sum_{C_k\in\mathcal{X}} \frac{r_k}{2} +\sum_{C_k\in\mathcal{Y}}\bigg(\frac{r_k}{2}+\frac{1}{2}\bigg)\leq \frac{h_i}{2}+\frac{s_i}{4}
\]
We repeat this construction for all iterations $i$ and let $\mathcal{C}\coloneq \bigcup_i \mathcal{C}_i$. It follows that 
\[|\mathcal{C}| \leq \sum_{i\ge 0} \bigg(\frac{h_i}{2}+\frac{s_i}{4}\bigg) \leq \frac12\sum_{i\ge 0}h_i+\frac14\sum_{i\ge 0} {s_i}\le\bigg(\frac32+\frac14\bigg)\sum_{i\ge 0}{s_i}\le \frac{7n}{4}.\]

\medskip

\noindent\underline{Separating the Edges of $H$.} It remains to separate the edges of $H$. Since $\{E(H_i)\}_{i\ge 0}$ partitions $E(H)$, and $\mathcal{C}_i$ is a path cover of $H_i$ for each $i$, it suffices to separate the edges within each $H_i$. For each iteration $i$, we construct a family of paths $\mathcal{P}_i$ that separate the edges of $H_i$. 

Fix some iteration $i$. If $H_i$ is empty, let $\mathcal{P}_i\coloneq \emptyset$. Otherwise, $H_i$ has at least two non isolated vertices. Define $\varphi$ as in Definition~\ref{Definition of phi}. By Lemma~\ref{Contract}(i), $P^*_i$ is a path on at most $2s_i$ vertices, and by Lemma~\ref{Contract}(iii), $H^*_i$ is a simple graph with $V(H^*_i)\subseteq V(P^*_i)$ and $E(H^*_i)\cap E(P^*_i)=\emptyset$. Theorem~\ref{Separate H} then yields a family $\mathcal{Q}^*_i$ of paths in $H^*_i\cup P^*_i$ that separate the edges of $H^*_i$ and $|\mathcal{Q}^*_i|$ is at most
\[
\frac{5|V(H^*_i)|}{2} \leq 5s_i.
\] 
By Lemma~\ref{Lifting}, each path $R^*\in \mathcal{Q}^*_i$ lifts to a simple path $R$ in $H_i\cup P_i$ satisfying $E(R)\cap E(H_i)=\varphi^{-1}(E(R^*)\cap E(H^*_i))$. Let $\mathcal{P}_i$ be the family of lifted paths so that $|\mathcal{P}_i|=|\mathcal{Q}^*_i|\leq 5s_i$. We claim that $\mathcal{P}_i$ separates $E(H_i)$. Indeed, choose $e,e'\in E(H_i)$ such that $e\not=e'$. By Lemma~\ref{Contract} (ii), $\varphi(e)$ and $\varphi(e')$ are distinct edges of $H^*_i$. Therefore, there exists some $R^*\in \mathcal{Q}^*_i$ that contains $\varphi(e)$ and avoids $\varphi(e')$. Let $R\in \mathcal{P}_i$ be the lift of $R^*$. By Lemma~\ref{Lifting}, $E(R)\cap E(H_i)=\varphi^{-1}(E(R^*)\cap E(H^*_i))$ hence $e\in R$ and $e' \notin R$. Symmetrically, there must also exist a path in $\mathcal{P}_i$ that contains $e'$, but avoids $e$. Therefore, $\mathcal{P}_i$ separates the edges of $H_i$. Let $\mathcal{P}\coloneq\bigcup_i\mathcal{P}_i$, so \[
|\mathcal{P}|\leq \sum_i 5s_i \leq 5n.
\]

\aftermath
\end{proof}

\subsection*{Statement on AI Use}

While preparing this manuscript, we used ChatGPT-6 Astra and Claude Opus 5.5 to assist with the exposition, refinement, and verification of mathematical arguments. ChatGPT-6 Astra helped draw the figures, and ChatGPT-6 Astra and Claude Opus 5.5 suggested edits, which we implement selectively. Finally, we originally obtained a longer proof of Corollary~\ref{d-degenerate} by extending the framework used in \cite{fernandes2026separating} for $2$-degenerate graphs. After prompting ChatGPT-6 Astra for feedback on that proof, the model suggested a shorter proof which we present in Section~\ref{Preliminatires}. The authors assume full responsibility for the content of this manuscript.

	\bibliographystyle{habbrv} 
	\bibliography{7.75n}

\end{document}